\documentclass[preprint,12pt]{elsarticle}
\usepackage{graphicx}
\usepackage{amssymb}
\usepackage{showlabels}
\usepackage{epsf}
\usepackage{amsbsy,amsmath}
\usepackage{mathtools}
\usepackage{mathrsfs}
\usepackage{amsfonts}
\usepackage{amssymb}
\usepackage{enumitem}
\usepackage{eucal}
\usepackage{cases}
\usepackage{graphics,mathrsfs}
\usepackage{amsthm}
\usepackage{secdot}
\usepackage{esint}
\usepackage{hyperref}
\usepackage{varwidth}
\usepackage{tasks}
\usepackage{geometry}
\newtheorem{theorem}{Theorem}[section]

\newtheorem{proposition}[theorem]{Proposition}

\theoremstyle{definition}
\newtheorem{definition}[theorem]{Definition}

\theoremstyle{remark}
\newtheorem{remark}[theorem]{Remark}
\numberwithin{equation}{section}

\numberwithin{equation}{section}
\usepackage{xcolor}

\begin{document}
\begin{frontmatter}
 	
 	
 	
 	\title{Algebraic topological techniques for elliptic problems involving fractional Laplacian}
  \author{A. Panda\fnref{label3}}
  \ead{akasmika44@gmail.com}
  	\author{D. Choudhuri\corref{cor1}\fnref{label3}}
  \ead{dc.iit12@gmail.com}
 	\author{A. Bahrouni\fnref{label2}}
 	\ead{bahrounianouar@yahoo.fr}
 
 	%
 	\cortext[cor1]{Corresponding author}
 	\fntext[label2]{Department of Mathematics, National Institute of Technology Rourkela}
 	\fntext[label3]{University of Monastir}
\begin{abstract}
	We prove the existence of infinitely many solutions to an elliptic problem by borrowing the techniques from algebraic topology. The solution(s) thus obtained will also be proved to be bounded.
	\begin{flushleft}
		{\bf Keywords}:~  Fractional Laplacian, Fractional Sobolev Space, homology group.\\
		{\bf AMS Classification}:~35R11, 35J75, 35J60, 46E35.
	\end{flushleft}
\end{abstract}
\end{frontmatter}
\section{Introduction}\label{Introduction}
We propose to study the following singular problem with a mixed operator:
\begin{align}
	\begin{split}\label{main prob}
		a(-\Delta)^su+b(-\Delta)u&=\lambda |u|^{-\gamma-1}u+\mu|u|^{2_s^*-2}u,~\text{in}~\Omega,\\
		u&=0,~\text{in}~\mathbb{R}^N\setminus\Omega,
	\end{split}
\end{align}
where $\Omega\subset\mathbb{R}^N (N\geq 2)$ is a bounded domain, $a,b\geq 0$, $0<s<1<2<2_s^*< 2^*$, $\lambda>0$, $\mu\in\mathbb{R}$, $\gamma\in (0,1)$,
and 
$$(-\Delta)^s u(x)=P.V.\int_{\mathbb{R}^N}\frac{(u(x)-u(y))}{|x-y|^{N+2s}}dy,~\forall ~x\in\Omega,$$
is the fractional Laplacian. We refer the operator $``a(-\Delta)^s+b(-\Delta)"$ as a mixed operator since it possesses both local as well as nonlocal features, and thus we refer the problems of kind $\eqref{main prob}$ as ``nonlocal-local'' elliptic problems. For a detailed study on this operator refer \cite{Biagi} and \cite{Dipierro}. The mixed operators with different order are nowadays gaining popularity in applied sciences, in theoretical studies and also in real world applications. The development of the literature includes viscosity solution methods \cite{Barles}, Cahn-Hilliard equations \cite{Caffarelli}, Aubry-Mather theory \cite{de la}, phase transitions \cite{Cabre}, probability and stochastics \cite{Mimica},  fractional damping effects \cite{Dell}, decay estimates for parabolic equations \cite{Dipierro1}, population dynamics \cite{Dipierro}, Bernstein-type regularity results \cite{Cabre1}.\\
We begin by considering the following two spaces $X$ and $Y$ which are the closures of $C_c^\infty(\Omega)$ in $H^{s}(\Omega)$ and $H^1(\Omega)$, respectively (refer Section $\ref{functional setting}$ for these notations). Let $Q=\mathbb{R}^{2N}\setminus((\mathbb{R}^N\setminus\Omega)\times(\mathbb{R}^N\setminus\Omega))$ and define the spaces as follows:
\begin{eqnarray}
X&=&\left\{u:\Omega\rightarrow\mathbb{R}:u~\text{is measurable and}~\iint_{Q}\frac{|u(x)-u(y)|^{2}}{|x-y|^{N+2s}}dxdy<\infty\right\}\nonumber\\
Y&=&\left\{u\in L^2(\Omega):|\nabla u|\in L^2(\Omega)\right\}.\nonumber
\end{eqnarray} 
The spaces $X,
Y$ are Banach spaces with respect to the following norms: 
\begin{align}
\begin{split}\label{nl_norm}
\|u\|_X&=\|u\|_2+\left(\iint_{Q}\frac{|u(x)-u(y)|^{2}}{|x-y|^{N+2s}}dydx\right)^{\frac{1}{2}}\\
&=\|u\|_2+[u]_{s,2},
\end{split}
\end{align}
and
\begin{align}
\begin{split}\label{l_norm}
\|u\|_Y&=\|u\|_2+\left(\int_{\Omega}|\nabla u|^2dx\right)^{\frac{1}{2}}\\
&=\|u\|_2+\|u\|_{1,2},
\end{split}
\end{align}
respectively. The norm $[\cdot]$ defined in \eqref{nl_norm} is the Gagliardo norm. Note that here $\|u\|_{\alpha}=(\int_{\Omega}|u|^{\alpha}dx)^{\frac{1}{\alpha}}$ for $1<\alpha<\infty$.\\
The study of the nonlocal-local elliptic problem in \eqref{main prob}, firstly directed our attention to fix a function space in which the solution(s) will be seeked for. Define the space $Z=\{u:a[u]_{s,2}^2+b\|u\|_{1,2}^2<\infty\}$ equipped with the following norm:
\begin{eqnarray}\label{norm}
\|u\|&=&(a[u]_{s,2}^2+b\|u\|_{1,2}^2)^{\frac{1}{2}}.
\end{eqnarray}
It is easy to see that the space $Z$ is a Banach space and also reflexive with respect to the norm $\|\cdot\|$, defined in $\eqref{norm}$.\\
The use of algebraic topological techniques to study problems having a singular nonlinearity is a rarity in the literature. Thus, the problem discussed here is new as the consideration of a singularity with a critical exponent with $\mu\in\mathbb{R}$ handled with Morse theoretic approach is not found anywhere in the literature to our knowledge. The question about the existence and multiplicity of positive weak solutions to problem $\eqref{main prob}$ with $\mu>0$ and with either $a=0$ or $b=0$ has been answered in \cite{Ghanmi,Ghosh,Giacomoni 2,Haitao,Hirano,Mukherjee} and the references therein. The authors of these works followed different tools such as variational method, concentration
compactness method, Nehari manifold method etc. but none of them used the techniques from algebraic topology to study $\eqref{main prob}$ with $a,b>0$ and $\mu\in \mathbb{R}$. \\ 
With the help of algebraic topological techniques, the existence and multiplicity results for the following problem have been established by many researchers under different growth conditions on the reaction term $f$ (either subcritical or critical growth conditions): 
\begin{align}
\begin{split}\label{prob s=1}
\mathcal{L}u&=f(x,u),~\text{in}~\Omega,\\
u&=0,~\text{on}~\partial\Omega.
\end{split}
\end{align}
 Problems of type $\eqref{prob s=1}$ were treated recently by Iannizzotto et al. in \cite{Iannizzotto} (finite multiplicity with $\mathcal{L}$ being the fractional $p$-Laplacian), Ferrara et al. in \cite{Ferrara} (at least one  non-trivial solution with $\mathcal{L}$ being a fractional integro-differential operator), D. Choudhuri in \cite{choudhuri} (infinitely many solutions with $\mathcal{L}$ being the $p(x)$-Laplacian) and the references therein. The paper by Papageorgiou \& R\u{a}dulescu \cite{papa2}, dealt with a nonlinear Robin problem and proved the multiplicity by producing three nontrivial solutions. The techniques thus differed from problem-to-problem addressed. \\
 The double phase problems of type $\eqref{prob s=1}$ with $\mathcal{L}$ being a $(p,q)$-Laplacian or a fractional $(p,q)$-Laplacian have been widely studied by many authors with different techniques. For instance, when $\mathcal{L}=(-\Delta_p-\Delta_q)$ with $p,q>1$, Gongbao \& Gao \cite{Gongbao}, Yin \& Yang \cite{Yin} used variational method, Liang et al. \cite{Liang} used Morse theoretical technique for the case $q=2\neq p$, and Marano et al. \cite{Marano}, Mugani \& Papageorgiou \cite{Mugani}  used the variational method with Morse theory and truncation comparison techniques. When $\mathcal{L}=((-\Delta_p)^s+(-\Delta_q)^s)$ with $p,q>1$ and $s\in (0,1)$, the problem $\eqref{prob s=1}$ has been discussed in \cite{Ambrosio,Bhakta,Goel,Isernia} using variational methods, and in \cite{Chen} using the Morse theory. With a combination of variational technique and Morse technique, N. S. Papageorgiou, V. D. R\u{a}dulescu, D. D. Repov\v{s} in \cite{papa1} established the existence (for $(p-1)$-superlinear case) and multiplicity (for $(p-1)$-linear resonant case) for Robin problems with $(p,q)$-Laplacian. For more details on double phase problems one can refer the recent piece of works by A. Bahrouni, V. D. R\v{a}dulescu \& D. D. Repov\v{s} \cite{AB3,AB4} and the bibliography therein.\\
Motivated by the former works, in this article, we study the singular problem $\eqref{main prob}$ using variational techniques and algebraic topological methods, specifically the  Morse theory and the critical groups (refer Section $\ref{functional setting}$). We establish the existence of infinitely many solutions to $\eqref{main prob}$ in Section $\ref{main results}$ followed by two subsections. Subsection $\ref{existence}$ and Subsection $\ref{multiplicity}$ deal with the existence part and the multiplicity part, respectively. In the Appendix, we will establish the boundedness result of these weak solutions. The books by Perera et al. \cite{pererabook} and Papageorgiou et al. \cite{papa0} are strongly recommended for a better understanding on the usage of Morse theory in the study of various elliptic PDEs. 
\section{Mathematical preliminaries}\label{functional setting}
\noindent A quintessential condition which the functional requires to satisfy is the the Palais-Smale condition (denoted by $(PS)$-condition) which is as follows.
\begin{definition}
	Let $X$ be a Banach space, and $I:X\rightarrow\mathbb{R}$ be a $C^1(X,\mathbb{R})$ functional. Given $c\in\mathbb{R}$ we say that the functional $I$ satisfies the Palais-Smale condition (or the $(PS)_c$-condition) at level $c$ if any bounded sequence $(u_n)\subset X$ such that $I(u_n)\rightarrow c$, and $I'(u_n)\rightarrow 0$ as $n\rightarrow\infty$ has a convergent subsequence in $X$.
\end{definition}
\noindent Below are the Sobolev embedding results that will be used throughout the article.
\begin{theorem}\label{poin}[\cite{Valdinoci}]
	Let $\Omega\subset\mathbb{R}^N$ be a bounded domain, $0<s<1$, and $2s<N$. Further, assume that $r\leq 2_s^*=\frac{2N}{N-2s}$. Then, there exists $C=C(r,s,N,\Omega)>0$ such that 
	$$\|u\|_{L^{r}(\Omega)}\leq C\|u\|_{X},~\forall u\in X.$$
	Moreover, this embedding is continuous for any $r\in [1,2_s^*]$, and compact for any $r\in [1,2_s^*)$. The above embedding holds also for $Z$.
\end{theorem}
\begin{theorem}[\cite{Valdinoci}]
	Let $\Omega\subset\mathbb{R}^N$ be a bounded domain, and $N>2$. Then, for every $\bar{r}\leq 2^*=\frac{2N}{N-2}$ there exists $\bar{C}=\bar{C}(r,N,\Omega)>0$ such that 
	$$\|u\|_{L^{\bar{r}}(\Omega)}\leq \bar{C}\|u\|_{Y},~\forall u\in Y.$$
	Moreover, this embedding is continuous for any $r\in [1,2^*]$, and compact for any $r\in [1,2^*)$. The above embedding holds also for $Z$.
\end{theorem}
\noindent We now present the fundamental tool that will be used to work with, namely the {\it homology theory} \cite{papa0}, which will be followed by the definition of {\it deformation} (see \cite{papa0}). 
\begin{definition}\label{defn1}
	A ``homology theory" on a family of pairs of spaces $(X,A)$ consists of:
	\begin{enumerate}
		\item A sequence $\{H_{k}(X,A)\}_{k\in \mathbb{N}_0}$  of abelian groups known as ``homology group'' for the pair $(X, A)$ (note that for the pair $(X,\phi)$, we write $H_k(X), k \in \mathbb{N}_0$).
		\item To every map of pairs $\varphi :(X,A)\rightarrow(Y,B)$ is associated a homomorphism
		$\varphi^{*} : H_k(X,A)\rightarrow H_k(Y,B)$ for all $k \in \mathbb{N}_0$.
		\item 	To every $k \in \mathbb{N}_0$ and every pair $(X,A)$ is associated a homomorphism
		$\partial : H_k(X,A)\rightarrow H_{k-1}(A)$ for all $k \in \mathbb{N}_0$.
	\end{enumerate}
	Here, $\mathbb{N}_0=\mathbb{N}\cup\{0\}$. 
	These objects satisfy the following axioms: 
	\begin{description}
		\item[($A_1$)] If $\varphi=id_{X}$, then $\varphi_{*}=id|_{H_k(X,A)}$.
		\item[($A_2$)] If $\varphi:(X,A)\rightarrow (Y,B)$, and $\psi:(Y,B)\rightarrow (Z,C)$ are maps of pairs, then $(\psi\circ\varphi)_{*}=\psi_{*}\circ\varphi_{*}$.
		\item[($A_3$)] If $\varphi:(X,A)\rightarrow(Y,B)$ is a map of pairs, then $\partial\circ\varphi_{*}=(\varphi|_{A})_{*}\circ\partial$.
		\item[($A_4$)] If $i:A\rightarrow X$ and $j:(X,\phi)\rightarrow(X,A)$ are inclusion maps, then the following sequence is exact
		$$...\xrightarrow[]{\partial} H_{k}(A)\xrightarrow[]{i_{*}} H_k(X)\xrightarrow[]{j_{*}} H_k(X,A)\xrightarrow[]{\partial} H_{k-1}(A)\rightarrow...$$
		Recall that a chain $...\xrightarrow[]{\partial_{K+1}} C_{k}(X)\xrightarrow[]{\partial_{k}} C_{K-1}(X)\xrightarrow[]{\partial_{k-1}} C_{k-2}(X)\xrightarrow[]{\partial_{k-2}} ...$ is said to be exact if $im(\partial_{k+1})=ker(\partial_k)$ for each $k\in\mathbb{N}_0$.
		\item[($A_5$)] If $\varphi, \psi:(X,A)\rightarrow(Y,B)$ are homotopic maps of pairs, then $\varphi_{*}=\psi_{*}$.
		\item[($A_6$)] (Excision): If $U\subseteq X$ is an open set with $\bar{U}\subseteq \text{int}(A)$, and $i:(X\setminus U,A\setminus U)\rightarrow (X,A)$ is the inclusion map, then $i_{*}:H_{k}(X\setminus U,A\setminus U)\rightarrow H_k(X,A)$ is an isomorphism.
		\item[($A_7$)] If $X=\{*\}$, then $H_k({*})=0$ for all $k\in\mathbb{N}$.
	\end{description}
\end{definition}
\begin{definition}\label{deformation}
	A continuous map $F: X\times [0,1]\to X$ is a deformation retraction of a space $X$ onto a subspace $A$ if, for every $x \in X$ and $a \in A$, $F(x,0)=x$, $F(x,1)\in A$, and $F(a,1)=a$.
\end{definition}
\noindent An important result in Morse theory is stated below.
\begin{theorem}\label{imp_lem_MT}
	Let $I\in C^2(X)$ satisfy the Palais-Smale condition, and let `$a$' be a regular value of $I$. Then, $H_*(X,I^a)\neq 0$, implies that $K_{I}\cap I^a\neq \emptyset$ where $$K_{I}=\{u\in X: I'(u)=0\},$$ and $$I^a=\{u\in X: I(u)\leq a\}.$$
\end{theorem}
\begin{remark}
	Another notation which will be used in the article is $$K_{I,D}=\{u\in X: I(u)\in D\}$$ where $D$ is a connected subset of $\mathbb{R}$.
\end{remark}
\begin{remark}\label{rem1}
	Prior to applying the Morse lemma we recall that for a Morse function the following holds:
	\begin{enumerate}
		\item $$H_{*}(I^c,I^c\setminus \text{Crit}(I,c))=\oplus_{j}H_*(I^c\cap N^j,(I^c\setminus\{x^j\})\cap N^j).$$
		\item \[H_{k}(I^c\cap N,I^c\setminus\{x\}\cap N)= \begin{cases}
		\mathbb{R}, & k=m(x) \\
		0, & \text{otherwise}
		\end{cases}\]
		where $m(x)$ is a Morse index of $x$, and $x$ is a critical point of $f$.\\
		\item Further $$H_k(I^a,I^b)=\oplus_{\{i:m(x_i)=k\}}\mathbb{Z}=\mathbb{Z}^{m_k(a,b)}$$ where $m_k(a,b)=n(\{i:m(x_i)=k,x_i\in K_{I,(a,b)}\})$. Here, $n(S)$ is the number of elements in the set $S$.
		\item Morse relation 
		$$\sum_{u\in K_{J,[a,b]}}\sum_{k\geq 0}dim C_k(I,u) t^k=\sum_{k\geq 0}dim H_k(J^{a},J^{b}) t^k+(1+t)Q_t$$ for all $t\in\mathbb{R}$.
	\end{enumerate}
\end{remark}
\noindent In this paper, we use the notion of local $(m, n)$-linking $(m,n\in \mathbb{N})$ (see Definition 2.3, \cite{papa1}) to prove the existence of solution.
\begin{definition}\label{linking_defn}
	Let $W$ be a Banach space, $I\in C^1(W,\mathbb{R})$, and $0$ an isolated critical point of $I$ with $I(0)=0$. Further, assume that $m,n\in \mathbb{N}$. We say that $I$ has a ``local $(m, n)$-linking'' near the origin if there exist a
	neighborhood U of 0, $E_0\neq\emptyset$, $E \subseteq U$, and $D\subseteq W$ such that $0\notin E_0\subseteq E$, $E_0\cap D=\emptyset$, and
	\begin{enumerate}
		\item 0 is the only critical point of $I$ in $I^0\cap U$, where $I^0=\{u\in W: I(u)\leq 0\}$,
		\item $Dim~im(i^*)-Dim~im(j^*)\geq n$, where
		$$i^*:H_{m-1}(E_0)\rightarrow H_{m-1}(W\setminus D)~\text{and}~j^*:H_{m-1}(E_0)\rightarrow H_{m-1}(E)$$
		are the homomorphisms induced by the inclusion maps $i: E_0\rightarrow W\setminus D$ and $j: E_0\rightarrow E$,
		\item $I|_{E}\leq 0\leq I|_{U\cap D\setminus\{0\}}$.
	\end{enumerate}
\end{definition}
\section{Proof of the main results}\label{main results}
\noindent Let $F(x,t)=\int_{0}^{t}f(s,x)ds$ be the primitive of $f(x,t)=\lambda|t|^{-\gamma-1}t+\mu |t|^{2_s^*-2}t$.\\
We say $u\in Z$ to be a weak solution of $\eqref{main prob}$ if for every $\varphi\in Z$, we have
\begin{align}
a\iint_{Q}\frac{(u(x)-u(y))}{|x-y|^{N+2s}}(\varphi(x)-\varphi(y)) dx dy\nonumber+b\int_{\Omega\time\Omega}\nabla u\cdot\nabla\varphi dx=&\int_{\Omega}\lambda|u|^{-\gamma-1}u\varphi(x)dx\nonumber\\&+\int_{\Omega}\mu |u|^{2_s^*-2}u\varphi(x)dx.\nonumber
\end{align}
Clearly, a weak solution to problem $\eqref{main prob}$ is a critical point of the corresponding energy functional
\begin{align}
\begin{split}\label{energy fnal}
I(u)=&\frac{a}{2}\iint_{Q}\frac{|u(x)-u(y)|^{2}}{|x-y|^{N+2s}}dxdy+\frac{b}{2}\int_{\Omega}|\nabla u|^2dx-\int_{\Omega}F(x,u)dx.
\end{split}
\end{align}
However, it is easy to see that the functional $I$ is not $C^1(Z)$ due to the presence of the singular term. Therefore, instead of working with the original functional $I$, we will use a cut-off functional, $\bar{I}$.
\subsection{Existence of solution using a local $(1,1)$ linking at 0}\label{existence}
We first define the functional $\bar{I}$. Let us consider the following cut-off function.
$$\xi(t)=\begin{cases}
1, ~\text{if}~ |t|\leq l\\
\xi ~\text{is decreassing, if}~ l\leq t\leq 2l\\
0,~\text{if}~ |t|\geq 2l.
\end{cases}$$
Let us now consider the following {\it cut-off} problem:
\begin{align}\label{main3}
\begin{split}
a(-\Delta)^su-b\Delta u&= \lambda\cfrac{u}{|u|^{\gamma+1}}+ \tilde{g}(u)~\text{in}~\Omega,\\
u&=0~\text{in}~\mathbb{R}^N\setminus\Omega,
\end{split}
\end{align}
where, $$\tilde{g}(u)=\mu|u|^{2_s^*-2}u\xi(\|u\|).$$
Let $\tilde{F}(x,t)=\int_{0}^{t}\tilde{f}(s,x)ds$ be the primitive of $\tilde{f}(x,t)=\lambda|t|^{-\gamma-1}t+\tilde{g}(t)$. We say $\tilde{u}\in Z$ to be a weak solution of $\eqref{main3}$ if for every $\varphi\in Z$, we have
\begin{align}
a\iint_{Q}\frac{(\tilde{u}(x)-\tilde{u}(y))}{|x-y|^{N+2s}}(\varphi(x)-\varphi(y)) dx dy\nonumber+b\int_{\Omega\time\Omega}\nabla \tilde{u}\cdot\nabla\varphi dx=&\int_{\Omega}\lambda|\tilde{u}|^{-\gamma-1}\tilde{u}\varphi(x)dx\nonumber\\&+\int_{\Omega}\mu \xi(\|\tilde{u}\|) |\tilde{u}|^{2_s^*-2}\tilde{u}\varphi(x)dx.\nonumber
\end{align}
Consequently, a weak solution to problem $\eqref{main3}$ is a critical point of the corresponding energy functional
\begin{align}
\begin{split}\label{energy3}
\bar{I}(u)=&\frac{a}{2}\iint_{Q}\frac{|u(x)-u(y)|^{2}}{|x-y|^{N+2s}}dxdy+\frac{b}{2}\int_{\Omega}|\nabla u|^2dx-\int_{\Omega}\tilde{F}(x,u)dx.
\end{split}
\end{align}
\begin{remark}\label{key_obs} One can easily see that if $u$ is a weak solution to \eqref{main3} with $\|u\|\leq l$, then $u$ is also a weak solution to \eqref{main prob}. \end{remark}
\begin{remark}\label{positivity}
It is easy to observe that $I(u)=I(|u|)$. Hence, if a solution exists it has to be nonnegative. Furthermore, due to the presence of the singular term the solution is forced to be positive a.e. in $\Omega$.
\end{remark}
\noindent We first verify that that functional $\bar{I}$ satisfies the (PS)- condition. Let us denote 
\begin{equation}\label{best}
S=\underset{u\in X\setminus\{0\}}{\inf}\frac{\|u\|_{X}^2}{\|u\|_{L^{2_s^*}(\Omega)}^2}
\end{equation}
which is the best Sobolev constant in the Sobolev embedding (Theorem $\ref{poin}$). 
\begin{theorem}
	The functional $\bar{I}$ satisfies the $(PS)_c$-condition for\\ $$c<c_*= \left(\frac{1}{2}-\frac{1}{2_s^*}\right)\mu^{\frac{-2}{2_s^*-2}}S^{\frac{2_s^*}{2_s^*-2}}-\left(\frac{1}{2}-\frac{1}{2_s^*}\right)^{-\frac{\gamma-1}{\gamma+1}}\left[|\Omega|^{\frac{2_s^*-1+\gamma}{2_s^*}}S^{-\frac{1-\gamma}{2}}\frac{\lambda}{1-\gamma}\right]^{\frac{2}{1+\gamma}},$$
	where $c_*>0$ for sufficiently small $\lambda,|\mu|>0$.
	\end{theorem}
	\begin{proof}
		We see that although the functional has been modified with a cut-off, yet the functional is not $C^1$ which is not in the premise of the Palais-Smale condition. However, we will devise a scheme to tackle this situation by picking the sequence in a way that the singularity at $0$ gets avoided since the functional is $C^1$ in $Z\setminus\{0\}$.  Suppose $(u_n)\subset Z$ is an eventually zero sequence, then apparently it converges to $0$, and we discard the sequence immediately. Suppose, $(u_n)\subset Z$ is a sequence with infinitely many terms of the sequence equal to $0$, then we will hand-pick a subsequence of $(u_n)$ with all non-zero terms. Thus, without loss of generality we will let $(u_n)$ such that $u_n\neq 0$ for every $n\in\mathbb{N}$. Let this sequence $(u_{n})$ be such that
		\begin{align}\label{ps 1}
		\bar{I}(u_{n})\rightarrow c,~\text{and}~ \bar{I}^{\prime}(u_{n})\rightarrow0
		\end{align} 
		as $n\rightarrow\infty$.
		Therefore, let us consider a sequence $(u_n)\subset Z$ such that $\bar{I}(u_n)\rightarrow c$ for some $c\in \mathbb{R}$, and $\bar{I}^\prime(u_n)\rightarrow 0$ as $n\rightarrow\infty$. It is not difficult to see that the subsequence is bounded in $Z$. This has the following consequences:
		\begin{align}\label{conv1}
		\begin{split}
		u_n&\rightharpoonup u_0~\text{in}~Z; \|u_n\|\rightarrow M,\\
		u_n&\rightarrow u_0~\text{in}~L^{r}(\Omega)~\text{for any}~1\leq r<2_{s}^*,\\
		u_n&\rightarrow u_0\neq 0~\text{a.e. in}~\Omega,
		\end{split}
		\end{align}
		as $n\rightarrow\infty$. Now, consider
		\begin{align}\label{PS_1_AUX}
		\begin{split}
		o(1)=&\langle \bar{I}'(u_n),u_n-u \rangle\\
		=&(\|u_n\|^2-\|u\|^2)-\lambda\int_{\Omega}|u_n|^{-\gamma-1}u_n(u_n-u)dx-\mu\int_{\Omega}\xi(\|u_n\|)|u_n|^{2_s^*-2}u_n(u_n-u)dx\\
		=&(\|u_n\|^2-\|u\|^2)-\mu\xi(M)\int_{\Omega}|u_n|^{2_s^*}dx+\mu\xi(M)\int_{\Omega}|u|^{2_s^*}dx+o(1)\\
		=&\|u_n-u\|^2-\mu\xi(M)\int_{\Omega}|u_n-u|^{2_s^*}dx+o(1).
		\end{split}
		\end{align}
		We obtain
		\begin{align}\label{PS_2_AUX}
		\begin{split}
		\underset{n\rightarrow\infty}{\lim}\|u_n-u\|^2&=\mu\xi(M)\underset{n\rightarrow\infty}{\lim}\int_{\Omega}|u_n-u|^{2_s^*}dx=\mu\xi(M)N^{2_s^*}.
		\end{split}
		\end{align}
		If $\mu\leq 0$, we produce a contradiction from $\eqref{PS_2_AUX}$, and we guarantee that $u_n\rightarrow u$ strongly in $Z$. Therefore, we now proceed for $\mu>0$. \\
		Further, if $N=0$, then we obtain $u_n\rightarrow u$ in $Z$ as $n\rightarrow\infty$, since $M>0$. Therefore, we will show that $N=0$. On the contrary let us assume that $N>0$. Thus, we get
		\begin{align}\label{PS_3_AUX}
		\begin{split}
		0\leq \underset{n\rightarrow\infty}{\lim}\|u_n-u\|^2&=\mu\xi(M)N^{2_s^*}.
		\end{split}
		\end{align}
		Next, from $\eqref{conv1}$, \eqref{PS_3_AUX}, and $\eqref{best}$ we have
		\begin{align}\label{PS_4_AUX}
		\begin{split}
		SN^2&\leq \mu\xi(M)N^{2_s^*}\leq\mu N^{2_s^*} \\
		M^2-\|u\|^2&=\mu\xi(M)N^{2_s^*}.
		\end{split}
		\end{align}
		It also follows from \eqref{PS_4_AUX} that
		\begin{align}\label{PS_5'_AUX}
		\begin{split}
		N&\geq \left(\frac{S}{\mu}\right)^{\frac{1}{2_s^*-2}},
		\end{split}
		\end{align}
		and 
		\begin{align}\label{PS_5_AUX}
		\begin{split}
		M^2\geq SN^2&\geq \mu^{\frac{-2}{2_s^*-2}}S(S^{\frac{2}{2_s^*-2}})\\
		&=\mu^{\frac{-2}{2_s^*-2}}S^{\frac{2_s^*}{2_s^*-2}}.
		\end{split}
		\end{align}
		 Consider 
		\begin{align}\label{contradiction}
		\begin{split}
		c+o(1)=&\bar{I}(u_n)-\frac{1}{2_s^*}\langle \bar{I}'(u_n),u_n\rangle\\
		=&a\left(\frac{1}{2}-\frac{1}{2_s^*}\right)[u_n]_{s,2}^2+b\left(\frac{1}{2}-\frac{1}{2_s^*}\right)\|\nabla u_n\|_2^2+\lambda\left(\frac{1}{2_s^*}-\frac{1}{1-\gamma}\right)\int_{\Omega}|u_n|^{1-\gamma}dx\\
		\geq& \left(\frac{1}{2}-\frac{1}{2_s^*}\right)\|u_n\|^2-\frac{\lambda}{1-\gamma}\int_{\Omega}|u_n|^{1-\gamma}dx.
		\end{split}
		\end{align}
		Therefore, by the Brezis-Lieb theorem, the Young's inequality, and the results derived in this theorem, we pass the limit $n\rightarrow\infty$ in $\eqref{contradiction}$ to obtain the following:
		\begin{align}\label{PS_9_AUX}
		\begin{split}
		c\geq& \left(\frac{1}{2}-\frac{1}{2_s^*}\right)(M^{2}+\|u\|^{2})-\frac{\lambda}{1-\gamma}\int_{\Omega}|u|^{1-\gamma}dx\\
		\geq&\left(\frac{1}{2}-\frac{1}{2_s^*}\right)(M^{2}+\|u\|^{2})-|\Omega|^{\frac{2_s^*-1+\gamma}{2_s^*}}S^{-\frac{1-\gamma}{2}}\frac{\lambda}{1-\gamma}\|u\|^{1-\gamma}\\
		\geq& \left(\frac{1}{2}-\frac{1}{2_s^*}\right)(M^{2}+\|u\|^{2})-\left(\frac{1}{2}-\frac{1}{2_s^*}\right)\|u\|^{2}-\left(\frac{1}{2}-\frac{1}{2_s^*}\right)^{-\frac{\gamma-1}{\gamma+1}}\left[|\Omega|^{\frac{2_s^*-1+\gamma}{2_s^*}}S^{-\frac{1-\gamma}{2}}\frac{\lambda}{1-\gamma}\right]^{\frac{2}{1+\gamma}}\\
		\geq& \left(\frac{1}{2}-\frac{1}{2_s^*}\right)M^{2} -\left(\frac{1}{2}-\frac{1}{2_s^*}\right)^{-\frac{\gamma-1}{\gamma+1}}\left[|\Omega|^{\frac{2_s^*-1+\gamma}{2_s^*}}S^{-\frac{1-\gamma}{2}}\frac{\lambda}{1-\gamma}\right]^{\frac{2}{1+\gamma}}\\
		\geq& \left(\frac{1}{2}-\frac{1}{2_s^*}\right)\mu^{\frac{-2}{2_s^*-2}}S^{\frac{2_s^*}{2_s^*-2}}-\left(\frac{1}{2}-\frac{1}{2_s^*}\right)^{-\frac{\gamma-1}{\gamma+1}}\left[|\Omega|^{\frac{2_s^*-1+\gamma}{2_s^*}}S^{-\frac{1-\gamma}{2}}\frac{\lambda}{1-\gamma}\right]^{\frac{2}{1+\gamma}}\\
		=&c_*
		\end{split}
		\end{align}
for sufficiently small $\lambda>0$ and $|\mu|>0$. This is a contradiction. 
		\end{proof}
\begin{theorem}\label{linking}
	The functional $\bar{I}$ has a local $(1,1)$ - linking at the origin.
\end{theorem}
\begin{proof} 
	We define $V=\mathbb{R}$. Clearly $V$ is a one dimensional vector subspace of $Z$. We now choose $\nu\in (0,1)$ small enough so that $K_{\bar{I}}\cap\overline{B_{\nu}(0)}=\{0\}$ where $B_{\nu}(0)=\{u\in Z:\|u\|<\nu\}$, and $K_{\bar{I}}=\{u\in Z: \bar{I}'(u)=0\}$. Define $$E=V\cap\overline{B_{\nu}(0)}$$
	for small enough $\nu\in(0,1)$. Note that any two norms are equivalent over $E$. Therefore, using this to our advantage, for a sufficiently small $\nu>0$ there exists $\delta>0$, we have 
	$$\|u\|\leq \nu\Rightarrow \|u\|_{\infty}\leq \delta~\text{for all}~u\in E.$$
	Therefore, 
	\begin{align}
	\begin{split}\label{ineq0}
	\bar{I}(u)=&\frac{a}{2}\iint_{Q}\frac{|u(x)-u(y)|^{2}}{|x-y|^{N+2s}}dxdy+\frac{b}{2}\int_{\Omega}|\nabla u|^2dx-\frac{\lambda}{1-\gamma}\int_{\Omega}|u|^{1-\gamma}dx-\frac{\mu}{2_s^*}\int_{\Omega}\xi(\|u\|)|u|^{2_s^*}dx\\
	=&\frac{1}{2}\|u\|^{2}-\frac{\lambda}{1-\gamma}\int_{\Omega}|u|^{1-\gamma}dx-\frac{\mu}{2_s^*}\int_{\Omega}\xi(\|u\|)|u|^{2_s^*}dx\\
	\leq&\frac{1}{2}\|u\|^{2}- \frac{\lambda}{1-\gamma}C\|u\|^{1-\gamma}-\frac{\mu}{2_s^*} C'\|u\|^{2_s^*}\leq 0\\
	\end{split}
	\end{align}
	where equivalence of norms followed from the set $E$ being in a finite dimensional space $V$. Further, define 
	$$D'=\left\{u\in Z:\|u\|^2\geq \frac{4\lambda}{1-\gamma}\|u\|_{1-\gamma}^{1-\gamma}\right\}.$$ 
	Thus, for any $u\in D'$,
	\begin{align}
	\begin{split}\label{ineq1}
	\bar{I}(u)=&\frac{a}{2}\iint_{Q}\frac{|u(x)-u(y)|^{2}}{|x-y|^{N+2s}}dxdy+\frac{b}{2}\int_{\Omega}|\nabla u|^2dx-\frac{\lambda}{1-\gamma}\int_{\Omega}|u|^{1-\gamma}dx-\frac{\mu}{2_s^*}\int_{\Omega}\xi(\|u\|)|u|^{2_s^*}dx\\=&\frac{1}{2}\|u\|^2-\frac{\lambda}{1-\gamma}\int_{\Omega}|u|^{1-\gamma}dx-\frac{\mu}{2_s^*}\int_{\Omega}|u|^{2_s^*}dx
	\\\geq&\frac{1}{4}\|u\|^{2}-\frac{|\mu|}{2_s^*} C\|u\|^{2_s^*}>0
	\end{split}
	\end{align}
	holds for small enough $\|u\|< \nu$. This holds for all $D=D'\cap (B_\nu(0)\setminus\{0\})$. Define $$E_0=V\cap\partial B_\nu(0).$$ Clearly $E_0\cap D=\emptyset$. Therefore, we arrive at the following $$I|_{E}\leq 0 < I|_{D}.$$
	Further, define
	$$h:[0,1]\times Z\setminus D'\rightarrow Z\setminus D'$$
	as $$h(t,u)=(1-t)u+t\cdot \nu\frac{v}{\|v\|}$$ where $v=\alpha\in V$. Clearly $0\in D'$ which makes the definition of $h(\cdot,\cdot)$ valid. Observe that 
	\begin{align}
	\begin{split}\label{ineq2}
	h(0,u)&=u,~\text{in}~Z\setminus D'\\
	h(1,u)&=\nu\frac{v}{\|v\|},~\text{in}~V\cap\partial B_\nu(0)=E_0.
	\end{split}
	\end{align}
	Thus, $E_0$ is a retract of $Z\setminus D'$. Therefore, 
	$$i^*:H_0(E_0)\rightarrow H_0(Z\setminus D)$$ is an isomorphism. Note that $E_0=\{v,-v\}$ for some $v\neq 0$. Therefore, from $Dim H_0(E_0)=2$ since $H_0(E_0)=\mathbb{R}\oplus\mathbb{R}$. Thus, $Dim~ im(i^*)=2$.\\
	Further, $E$ is an interval $[-\alpha,\alpha]$ which is contractible to a point. Thus, $H_0(E)=\mathbb{R}$. Hence, if $$j^*:H_0(E_0)\rightarrow H_0(E),$$ then $Dim~im(i^*)-Dim~im(j^*)=2-1=1$. Thus, the hypothesis of the Definition \ref{linking_defn} is satisfied, and hence there exists a local $(1,1)$ linking at 0. 
\end{proof}
\begin{proposition}\label{Ckiszero}
	$C_k(\bar{I},\infty)=0$ for all $k\in\mathbb{N}$.
\end{proposition}
\begin{proof}
	Let $t>0$, and consider
	\begin{align}
	\begin{split}\label{ineq3}
	\frac{d}{dt}\bar{I}(tu)=&\langle \bar{I}'(tu),u\rangle\\
	=&\frac{1}{t} \langle \bar{I}'(tu),tu\rangle\\
	=& \frac{1}{t}\left[a \iint_{Q}\frac{|tu(x)-tu(y)|^{2}}{|x-y|^{N+2s}}dxdy+b\int_{\Omega}|\nabla tu|^2dx\right.\\
	&\left.-\int_{\Omega}\tilde{f}(x,tu)tudx\right]\\
	=&   \frac{1}{t}\left[a \iint_{Q}\frac{|tu(x)-tu(y)|^{2}}{|x-y|^{N+2s}}dxdy+b\int_{\Omega}|\nabla tu|^2dx\right.\\
	&\left.-\lambda\int_{\Omega}|tu|^{1-\gamma}dx-\mu\int_{\Omega}\xi(\|tu\|)|tu|^{2_s^*}dx\right]
	\end{split}
	\end{align}
	Observe that for large $t>0$ we have $\bar{I}(tu)\geq \mathbb{J}_0>c$, for some $c>0$, since $\bar{I}(tu)\rightarrow\infty$ as $t\rightarrow\infty$. Therefore,
	\begin{eqnarray}\label{der_neg} 
	\frac{d}{dt}\bar{I}(tu)&>&0
	\end{eqnarray} for sufficiently large $t>0$. Thus, there exists a unique $h(u)>0$ such that $\bar{I}(h(u)u)=\mathbb{J}_0$. This actually implies that $h\in C(\partial B_1)$. On extending $h(\cdot)$ to $Z\setminus\{0\}$ by defining $$\tilde{h}(u)=\frac{1}{\|u\|}h\left(\frac{u}{\|u\|}\right)$$ for all $u\in  Z\setminus\{0\}$, we obtain $\tilde{h}\in Z\setminus\{0\}$, and $\bar{I}(\tilde{h}(u)u)=\mathbb{J}_0$. Also if $\bar{I}(u)=\mathbb{J}_0$ then $\tilde{h}(u)=1$. Therefore, we define
	\begin{numcases}{\hat{h}(u)=}
	1, & if $\bar{I}(u)\geq\mathbb{J}_0$\nonumber\\
	\tilde{h}(u), & if $\bar{I}(u)\leq\mathbb{J}_0$\nonumber
	\end{numcases}
	which makes $\hat{h}$ continuous. Further define 
	$$g(t,u)=(1-s)u+s\hat{h}(u)u$$ for all $(t,u)\in[0,1]\times(Z\setminus\{0\})$.
	Thus, we have $$g(0,u)=u,~g(1,u)=\tilde{h}(u)u\in \bar{I}^{\mathbb{J}_0},$$
	and 
	$$g(t,.)|_{\bar{I}^{\mathbb{J}_0}}=id|_{\bar{I}^{\mathbb{J}_0}}$$
	for all $t\in[0,1]$. What we conclude from here is that $\bar{I}^{\mathbb{J}_0}$ is a deformation of $Z\setminus\{0\}$. By standard definition of a homotopy, it is easy to see that $\partial B_1=\{u\in \bar{Z}:\|u\|=1\}$ is a deformation of $Z\setminus\{0\}$. Thus, on choosing $\mathbb{J}_0$ sufficiently negative we get 
	$$C_k(\bar{I},\infty)=H_k(\bar{Z},\partial B_1)=0$$ for all $k\in \mathbb{N}$.
\end{proof}
\begin{remark}\label{C1geq1} We use a variant of the result - {\it If $X$ is a Banach space, $\bar{I}\in C^1(X)$, $0\in K_{I}$ is isolated, and $\bar{I}$ has a local $(m,n)$-linking near the origin, then rank $C_m(\bar{I},0)\geq n$.}
	Since we obtained a $(1,1)$ linking we conclude that rank of $C_1(\bar{I},0)\geq 1$. In our case $0$ is not a critical point of $\bar{I}$, however one can still construct $C_1(\bar{I},0)$ owing to not only the fact that $\bar{I}$ is well-defined at $0$ but also having a `{\it sharp}' isolated non regularity there.\end{remark}
\begin{theorem}\label{existence_one_soln}There exists a solution, say $u_0$, to the problem \eqref{main3}.
\end{theorem}
\begin{proof}
	As seen in the Remark $\ref{C1geq1}$ we have that rank $C_1(\bar{I},0)\geq 1$. In tandem with this and Proposition $\ref{Ckiszero}$, we are in a position to apply the Proposition $6.2.42$ (refer Theorem $\ref{papathm}$ in Appendix) of \cite{papa0} that guarantees the existence of a $u_0\in Z$ such that 
	$$u_0\in K_{\bar{I}}\setminus\{0\},$$ which further implies that $u_0\in Z\cap L^{\infty}(\Omega)$ is a solution of \eqref{main3}. For the boundedness of $u_0$, please refer the Theorem $\ref{bddness'}$ in the Appendix. Thus, the existence of a nontrivial, bounded solution to \eqref{main3} has been established. 
\end{proof}
\subsection{Multiplicity results}\label{multiplicity}
\begin{theorem}\label{twosoln}
	The problem \eqref{main3} has at least two nontrivial solutions in $Z\cap L^{\infty}(\Omega).$
\end{theorem}
\begin{proof}
	From the Theorem $\ref{existence_one_soln}$ it can be assumed that there is one nontrivial solution to \eqref{main3}. Let us further assume that there is exactly one nontrivial solution to \eqref{main3}.  We at first show that $H_k(Z,\bar{I}^{-a})$ for all $k\geq 0$. Pick a $u\in\{v\in Z:\|v\|=1\}=\partial B^{\infty}$, where $B^{\infty}=\{v\in Z:\|v\|\leq 1\}$. We make use of the equation \eqref{ineq1} which explains that for small enough $t>0$ we have $\bar{I}(tu)>0$. The functional $\bar{I}$ being $C^1$ in $Z\setminus\{0\}$ indicates that $\bar{I}'(tu)>0$ for this small $t>0$. Also from \eqref{der_neg} and in combination with the fact that the functional is superlinear, we have that for large enough $t>0$, $\bar{I}'(tu)<0$. Thus, there exists a unique $t(u)$ such that $\bar{I}'(t(u)u)=0$ since due to our assumption that there exists exactly one nontrivial solution. We can thus say that there exists a $C^1$-function $T:Z\setminus\{0\}\rightarrow\mathbb{R}^{+}$ defined by $u\mapsto t(u)$. We now define a standard  deformation retract $\eta$ of $Z\setminus B_r(0)$ into $\bar{I}^{-a}$ as follows (refer Definition \ref{deformation}). 
	\[\eta(s,u)=\begin{cases}
	(1-s)u+sT\left(\frac{u}{\|u\|}\right)\frac{u}{\|u\|}, & \|u\|\geq \nu, \bar{I}(u)\geq -a\\
	u, & \bar{I}(u)\leq -a.
	\end{cases}\] 
	It is not difficult to see that $\eta$ is a $C^1$ function over $[0,1]\times Z\setminus B_r(0)$. On using the map $\delta(s,u)=\frac{u}{\|u\|}$, for $u\in Z\setminus B_r(0)$ we claim that $H_k(Z,Z\setminus B_r(0))=H_k(B^{\infty},S^{\infty})$ for all $k\geq 0$. This is because, $H_k(B^{\infty},S^{\infty})\cong H_k(*,0)$. From elementary computation of homology groups with two $0$-dimensional simplices it is easy to see that $H_k(*,0)=\{0\}$ for each $k\geq 0$. 
	A result in \cite{pererabook} tells us that
	\[C_n(\bar{I},0)=\begin{cases}
	\mathbb{R}, & \text{if}~m(0)=N\\
	0, & \text{otherwise}.
	\end{cases}\] 
	Clearly, $m(0)\geq 2$. Therefore, from the Morse relation in the Remark $\ref{rem1}-4$ and the result above taken from \cite{pererabook}, we have for $a>0$
	\begin{align}
	\sum_{u\in K_{\bar{I},[-a,\infty)}}\sum_{k\geq 0}dim C_k(\bar{I},u)t^k&=t^{m(0)}+p(t)
	\end{align}   
	where $m(0)$ is the Morse index of $0$, and $p(t)$ contains the rest of the powers of $t$ corresponding to the other critical points, if any. On further using the the Morse relation we obtain 
	\begin{align}
	t^{m(0)}+p(t)&=(1+t)Q_t.
	\end{align}  
	This is because the $H_k$s are all trivial groups. Hence, $Q_t$ either contains $t^{m(0)}$ or $t^{m(0)-1}$ or both. Thus, there exists at least two nontrivial $u\in K_{\bar{I},[-a,\infty)}$ with $2\leq m(0)\leq N+1$. 
\end{proof}\begin{theorem}\label{subcrit}
	Let $\Omega$ be as above.Then, $\bar{I}$ has infinitely many critical points in $Z$. 
\end{theorem}
\begin{proof}
	We appeal to the Morse theory again from which we obtain the following. 
	\begin{align}
	\sum_{u\in K_{\bar{I},[-a,\infty)}}\sum_{k\geq 0}dim C_k(\bar{I},u)t^k&=t^{m(0)}+2\sum_{k\geq 0}\alpha_k t^k.
	\end{align} 
	The factor $2$ is due to the fact that if $u$ is a critical point then $(-u)$ is also a critical point. $\alpha_k$s are nonnegative integers. As in the proof of the Theorem $\ref{twosoln}$, we have $H_k(Z,\bar{I}_{-a}^{\infty})=0$. Therefore, we have the following identity over $\mathbb{R}$.
	\begin{align}\label{iden1}
	t^{m(0)}+2\sum_{k\geq 0}\alpha_k t^k&=(1+t)Q_t.
	\end{align}
	In particular for $t=1$ we have $1+2A=2B$, where the series on the left of the identity in \eqref{iden1} for $t=1$ is denoted by $A$, and the term on the right of the same for $t=1$ is denoted by $B$. This is possible only when the sum is infinite as finite sum leads to a contradiction that there exists an odd and an even number that agree. Thus, there exists infinitely many solutions.
\end{proof}
\noindent The following is the main existence result for the problem $\eqref{main prob}$.
\begin{theorem}\label{main theorem}
	The problem $\eqref{main prob}$ admits infinitely many nontrivial solutions in $Z\cap L^\infty(\Omega)$.
	\begin{proof}
		According to Theorem $\ref{existence_one_soln}-\ref{subcrit}$, the problem $\eqref{main3}$ has infinitely many solutions in $Z\cap L^\infty(\Omega)$. With a suitable choice of $l$ and by Remark $\ref{positivity}$, we conclude our proof.
	\end{proof}
\end{theorem}
\section{Appendix}\label{boundedness}
\begin{theorem}\label{papathm}
	(Theorem $6.2.42$ of Papagiorgiou et al. \cite{papa0})	If $X$ is a Banach space, $I\in C^1(X)$, $I$ satisfies the $Ce$-condition, $K_I$ is finite with $0\in K_I$, and for some $k\in\mathbb{N}$ we have $C_k(I,0)\neq 0$, $C_k(I,\infty)=0$, then there exists a $u\in K_I$ such that
	$\bar(u)<0$, $C_{k-1}(I,u)\neq 0$ or $I(u)>0$, and $C_{k+1}(I,u)\neq 0$.
\end{theorem}
\begin{theorem}\label{bddness'}
	Any solution to \eqref{main3} is in $L^{\infty}(\Omega)$.
\end{theorem}
\begin{proof} 
		The argument sketched here is a standard one, and hence we shall  only show that an improvement in integrability is possible up to $L^{\infty}$ assuming an integrability of certain order, say $p>1$. The boundedness follows from a {\it bootstrap} argument. Without loss of generality we can consider the set $\Omega'=\{x\in\Omega:u(x)>1\}$, and thus by the positivity of a fixed solution (refer Remark $\ref{positivity}$), say, $u$ we have $u=u^+>0$ a.e. in $\Omega$. Let $u\in L^{p}(\Omega)$ for $p>1$. Let $a=\max\{\lambda,|\mu|\}$. On testing with $u^p$ to obtain the following:	
		\begin{align}\label{bddness''}
		\begin{split}
		C\|u^{\frac{p+1}{2}}\|_{2_s^*}^{2}&\leq\|u^{\frac{p+1}{2}}\|\\
		&\leq\left(\lambda\int_{\Omega'}|u|^{p-\gamma}dx+\mu\int_{\Omega'}|u|^{2_s^*-1+p}dx\right)\frac{(p+1)^2}{4p}\\
		&\leq a\left(\int_{\Omega'}|u|^{p}(1+|u|^{2_s^*-1})dx\right)\frac{(p+1)^2}{4p};~\text{since in}~\Omega'~\text{we have}~u>1\\
		&\leq 2a\left(\int_{\Omega'}|u|^{p}|u|^{2_s^*}dx\right)\frac{(p+1)^2}{4p}\\
		&\leq 2a C''\|u\|_{\beta^*}^{2_s^*}\|u^p\|_{\mathfrak{t}};~\text{since by using H\"{o}lder's inequality}.
		\end{split}
		\end{align}
		Here, $\mathfrak{t}=\frac{\beta^*}{\beta^*-2_s^*}$ for some $\beta^*>1$, $\mathfrak{t}^*=\frac{tN}{N-ts}<2_{s}^*$. Thus, we also have
	\begin{align}\label{bddness1}
		\begin{split}
		C'\|u^{\frac{p}{2}}\|_{\beta^*}^2&\leq C'\|u^{\frac{p+1}{2}}\|_{\beta^*}^2\leq \|u^{\frac{p+1}{2}}\|_{2_s^*}^2.
		\end{split}
		\end{align}
		From the story so far, we know the following
		\begin{align}
		C'\|u^{\frac{p}{2}}\|_{\beta^*}^2&\leq 2a C''\|u\|_{\beta^*}^{2_s^*}\|u^p\|_{\mathfrak{t}}.
		\end{align}
		For the fixed $\beta^*>1$, we set $\eta=\frac{\beta^*}{2\mathfrak{t}}>1$ for a suitable choice of $t$, and $\tau=\mathfrak{t}p$ to get
		\begin{align}\label{moser1}
		\|u\|_{\eta\tau}&\leq C^{\frac{\mathfrak{t}}{\tau}}\|u\|_{\tau};~\text{where}~C=2a  C''\|u\|_{\beta^*}^{\alpha^+}~\text{is a fixed  quantity for a fixed solution}~u.\\
		\end{align}
		Let us now iterate with $\tau_0=\mathfrak{t}$, $\tau_{n+1}=\eta\tau_n=\eta^{n+1}\mathfrak{t}$. After $n$ iterations, the inequality $\eqref{moser1}$ yields
		\begin{align}\label{moser2}
		\|u\|_{\tau_{n+1}}&\leq C^{\sum\limits_{i=0}^{n}\frac{\mathfrak{t}}{\tau_i}}\prod\limits_{i=0}^{n}\left(\frac{\tau_i}{\mathfrak{t}}\right)^{\frac{\mathfrak{t}}{\tau_i}}\|u\|_{\mathfrak{t}}.
		\end{align}
		By using $\eta>1$ and the method of iteration, i.e. $\tau_0=\mathfrak{t}$, $\tau_{n+1}=\eta\tau_n=\eta^{n+1}\mathfrak{t}$,
		we have $$\sum\limits_{i=0}^{\infty}\frac{\mathfrak{t}}{\tau_i}=\sum\limits_{i=0}^{\infty}\frac{1}{\eta^i}=\frac{\eta}{\eta-1},$$ and
		$$\prod\limits_{i=0}^{\infty}\left(\frac{\tau_i}{\mathfrak{t}}\right)^{\frac{\mathfrak{t}}{\tau_i}}=\eta^{\frac{\eta^2}{(\eta-1)^2}}.$$
		Hence, on passing the limit $n\rightarrow\infty$ in \eqref{moser2}, we end up getting
		\begin{align}\label{moser3}
		\|u\|_{\infty}&\leq C^{\frac{\eta}{\eta-1}}\eta^{\frac{\eta^2}{(\eta-1)^2}}\|u\|_{\mathfrak{t}}.
		\end{align}
		Thus, $u\in L^{\infty}(\Omega)$.
\end{proof}
\section{References}



\begin{thebibliography}{99}
	\bibitem{Ambrosio}
	V. Ambrosio and T. Isernia. On a fractional $p\&q$ Laplacian problem with critical
	Sobolev-Hardy exponents, {\em Mediterranean Journal of Mathematics}, 15(6):219, 2018.
	
	\bibitem{AB3}
	A. Bahrouni, V. D. R\v{a}dulescu, D. D. Repov\v{s}, A weighted anisotropic variant of the Caffarelli-Kohn-Nirenberg
	inequality and applications, {\em Nonlinearity}, 31 (4), 1516-1534, 2018.
	
	\bibitem{AB4}
	A. Bahrouni, V. D. R\v{a}dulescu, D.D. Repov\v{s}, Double phase transonic flow problems with variable growth: non-linear patterns and stationary waves, {\em Nonlinearity}, 32 (7), 2481-2495, 2019.

	\bibitem{Barles}G. Barles, C. Imbert, Second-order elliptic integro-differential equations: viscosity solutions' theory revisited,
	{\em Ann. Inst. H. Poincar\'{e} Anal. Non Lin\'{e}aire}, 25 (3), 567-585, 2008.
	
	\bibitem{Bhakta}
	M. Bhakta, D. Mukherjee, et al. Multiplicity results for $(p;q)$ fractional elliptic equations
	involving critical nonlinearities, {\em Advances in Differential Equations}, 24(3/4), 185-228, 2019.
	
	\bibitem{Biagi}S. Biagi, S. Dipierro, E. Valdinoci and E. Vecchi, Mixed local and nonlocal elliptic operators: regularity and maximum principles, arxiv:2005.06907v2 [math.AP] 20 May 2020.
	
	\bibitem{Cabre1}X. Cabr\'{e}, S. Dipierro, E. Valdinoci, The Bernstein technique for integro-differential equations,  arXiv:2010.00376v1 [math.AP] 1 Oct 2020.
	
	\bibitem{Cabre}X. Cabr\'{e}, J. Serra, An extension problem for sums of fractional Laplacians and 1-D symmetry of phase transitions,
	{\em Nonlinear Anal.} 137, 246-265, 2016.
	
	\bibitem{Caffarelli}L. Caffarelli, E. Valdinoci, A priori bounds for solutions of a nonlocal evolution PDE, Analysis and numerics
	of partial differential equations, 141–163, Springer INdAM Ser., 4, Springer, Milan, 2013.
	
	\bibitem{Chen} F. Chen, Y. Yang, Existence results for fractional $(p, q)$-Laplacian equations via
	Morse theory, DOI- 10.13140/RG.2.2.30120.80642, 2020.
	
	\bibitem{choudhuri}
	D. Choudhuri, Understanding the existence of multiple solutions using the Morse theory -- a
	short note, {\em arXiv:2003.04239 [math.AP]}, 9 Mar 2020.
	
	\bibitem{de la}R. de la Llave, E. Valdinoci, A generalization of Aubry-Mather theory to partial differential equations and
	pseudo-differential equations, {\em Ann. Inst. H. Poincar\'{e} Anal. Non Lin\'{e}aire}, 26 (4), 1309-1344, 2009.
	
	\bibitem{Dell}F. Dell'Oro, V. Pata, Second order linear evolution equations with general dissipation, arxiv:1811.07667 [math.AP] 19 Nov 2018.
	
	\bibitem{Dipierro}S. Dipierro, E. Proietti Lippi and E. Valdinoci, (Non)local logistic equations with Neumann conditions,  arXiv:2101.02315v1 [math.AP] 7 Jan 2021.
	
	\bibitem{Dipierro1}S. Dipierro, E. Valdinoci, V. Vespri, Decay estimates for evolutionary equations with fractional time-diffusion,
	{\em J. Evol. Equ.}, 19 (2), 435-462, 2019.
	
	\bibitem{Ferrara}
	M. Ferrara, G. M. Bisci and B. Zhang, Existence of weak solutions for non-local fractional problems via Morse theory, \newblock {\em Discrete Contin. Dyn. Syst. Ser. B}, 19 (8), 2493-2499, 2014.
	
	\bibitem{Ghanmi}A. Ghanmi and K. Saoudi, The Nehari manifold for a singular elliptic equation involving
	the fractional Laplace operator, {\em Fractional Differential Calculus}, 6(2), 201-217, 2016.
		
	\bibitem{Ghosh} K. Saoudi, S. Ghosh and D. Choudhuri, Multiplicity and H\"{o}lder regularity of solutions
	for a nonlocal elliptic PDE involving singularity, {\em J. Math. Phys.}, 60, 101509 1-28, 2019.
	
	\bibitem{Giacomoni 2}J. Giacomoni, T. Mukherjee and K. Sreenadh, Positive solutions of fractional elliptic
	equation with critical and singular nonlinearity, {\em Adv. Nonlinear Anal.}, 6 (3), 327-354,
	2017.
	
	\bibitem{Goel}
	D. Goel, D. Kumar, and K. Sreenadh. Regularity and multiplicity results for fractional
	$(p;q)$-Laplacian equations. {\em Communications In Contemporary Mathematics},
	2019. https://doi.org/10.1142/S0219199719500652
	
	\bibitem{Gongbao}
	L. Gongbao and Z. Guo. Multiple solutions for the $p\&q$-Laplacian problem with critical
	exponent, {\em Acta Mathematica Scientia}, 29(4), 903-918, 2009.
	
	\bibitem{Haitao}Y. Haitao, Multiplicity and asymptotic behavior of positive solutions for a singular semilinear
	elliptic problem, {\em J. Differential Equations}, 189, 487-512, 2003.
	
	\bibitem{Hirano}N. Hirano, C. Saccon and N. Shioji, Existence of multiple positive solutions for singular
	elliptic problems with concave and convex nonlinearities, {\em Adv. Differential Equations}, 9,
	12, 197-220, 2004.
	
	\bibitem{Iannizzotto}
	A. Iannizzotto, S. Liu, K. Perera, M. Squassina, Existence results for fractional $p$-Laplacian problems via Morse theory, {\em arXiv:1403.5388v1 [math.AP]} 21 Mar 2014.
	
	\bibitem{Isernia}
	T. Isernia. Fractional $p\&q$-Laplacian problems with potentials vanishing at infinity.
	{\em Opuscula Mathematica}, 40(1), 93-110, 2020.
	
	\bibitem{Valdinoci}
E. Di Nezza, G. Palatucci and E. Valdinoci, Hitchhiker's guide to the fractional Sobolev spaces, Bulletin des sciences math\'{e}matiques, 136(5), 521-573, 2012.
	
	\bibitem{Liang}
	Z. Liang, Y. Song, and J. Su, Existence of solutions to $(2;p)$-Laplacian equations by
	Morse theory, {\em Electronic Journal of Differential Equations}, 2017, 1-9, 2017.
	
	\bibitem{Marano}
	S. A. Marano, S. J. N. Mosconi, and N. S. Papageorgiou. Multiple solutions to $(p;q)$-
	Laplacian problems with resonant concave nonlinearity, {\em Advanced Nonlinear Studies},
	2015. https://doi.org/10.1515/ans-2015-5011.
	
	\bibitem{Mimica} A. Mimica, Heat kernel estimates for subordinate Brownian motions, {\em Proc. Lond. Math. Soc.} (3) 113 (5), 627–648, 2016.
	
	\bibitem{Mugani}
	D. Mugnai and N. S. Papageorgiou, Wang's multiplicity result for superlinear $(p;q)$-
	equations without the Ambrosetti-Rabinowitz condition. {\em Transactions of the American
		Mathematical Society}, 366(9), 2014.
	
	\bibitem{Mukherjee}T. Mukherjee and K. Sreenadh, Fractional elliptic equations with critical growth and
	singular nonlinearities, {\em Electronic Journal of Differential Equations}, 2016(54), 1-23, 2016.
	
	\bibitem{papa2}
	N. S. Papageorgiou, V. D. R\u{a}dulescu, \newblock Nonlinear Nonhomogeneous Robin Problems
	with Superlinear Reaction Term, \newblock {\em Adv. Nonlinear Stud.}, 16 (4), 737-764, 2016.
	
	\bibitem{papa0}
	N. S. Papageorgiou, V. D. R\u{a}dulescu, D. D. Repov\v{s},
	\newblock Nonlinear Analysis-Theory and Methods,
	\newblock {\em Springer Monographs in Mathematics, Springer nature, Cham}, 2019.
	
	\bibitem{papa1}
	N. S. Papageorgiou, V. D. R\u{a}dulescu, D. D. Repov\v{s},
	\newblock Existence and multiplicity of solutions for double-phase Robin problems,
	\newblock {\em arXiv:2006.01454v1[math.AP]}, 2020.
	
	\bibitem{pererabook}
	K. Perera, R. P. Agarwal and D. O'Regan,
	\newblock  Morse Theoretic Aspects of $p$-Laplacian Type Operators,
	\newblock {\em Mathematical surveys and Monographs, Amer. Math. Soc.}, 161, 2010.

\bibitem{Yin}
H. Yin and Z. Yang. Multiplicity of positive solutions to a $p-q$-Laplacian equation
involving critical nonlinearity, {\em Nonlinear Analysis: Theory, Methods \& Applications},
75(6), 3021-3035, 2012.
\end{thebibliography}
\end{document}